\documentclass[reqno,b5paper]{amsart}
\usepackage{amsfonts,amsmath,amssymb,amscd,amstext,amsbsy,latexsym}
\usepackage{amsthm}
\usepackage{enumerate}
\usepackage[mathscr]{eucal}
\usepackage{graphicx}
\newtheorem{theorem}{Theorem}[section]
\newtheorem{lemma}[theorem]{Lemma}
\newtheorem{proposition}[theorem]{Proposition}
\newtheorem{corollary}[theorem]{Corollary}

\theoremstyle{definition}
\newtheorem{definition}[theorem]{Definition}

\theoremstyle{definition}

\numberwithin{equation}{section}

\title[$\lambda$-perfect maps]{$\lambda$-perfect maps}
\author[M. Namdari \and M.A. Siavoshi]{M. Namdari* \and M.A. Siavoshi**}

\address{\llap{*\,}Department of Mathematics, Shahid Chamran University, Ahvaz, Iran}
\email{namdari@ipm.ir}
\address{\llap{**\,}Department of Mathematics, Shahid Chamran University, Ahvaz, Iran}
\email{m.siavoshi@scu.ac.ir}


\begin{document}

%\title{perfect maps}

\maketitle

\begin{abstract}
The $\lambda$-perfect maps, a generalization of perfect maps (continuous closed maps with compact fibers)  are presented. Using $P_\lambda$-spaces and  the concept of $\lambda$-compactness  some results regarding $\lambda$-perfect maps will be investigated.

\end{abstract}

\section{Introduction}
\baselineskip=0.6cm

A perfect map  is a kind of continuous function between
topological spaces. Perfect maps are weaker than homeomorphisms, but strong enough to preserve some
topological properties such as local compactness that are not always preserved by continuous maps.
Let $X, Y$ be two topological spaces such that $X$ is Hausdorff. A continuous map $f:X\to Y$ is said to be a perfect map provided that $f$ is closed,  surjective and each fiber $f^{-1}(y)$ is compact in $X$. In \cite{KNS} the authors are introduced  a generalization of compactness, i.e., $\lambda$-compactness. Motivated by this concept we are led to generalize the concept of perfect mapping.  In the current section we recall some preliminary definitions and related results such as $\lambda$-compact spaces and $P_\lambda$-spaces  which are presented in \cite{KNS}. In Section 2, we will introduce $\lambda$-perfect maps and generalize some classical results which are related to the perfect maps and finally we prove that if the composition of two continuous mappings is $\lambda$-perfect, then its components are $\mu$-compact and $\beta$-compact, where $\mu, \beta\leq\lambda$.

\begin{definition}\label{def 0-1}
  A topological space $X$ (not necessarily Hausdorff) is said to be $\lambda$-compact whenever each open cover of $X$ can be reduced to an open cover of $X$ whose cardinality is less than $\lambda$, where $\lambda$ is the least infinite cardinal number with this property.  $\lambda$ is called the compactness degree of $X$ and we write $d_c(X)=\lambda$. 
\end{definition} 

We note that compact spaces, Lindel\"{o}f non-compact spaces are $\aleph_\circ$-compact, $\aleph_1$-compact spaces, respectively and in general every topological space $X$ is $\lambda$-compact for some infinite cardinal number $\lambda$, see the concept of  Lindel\"{o}f number in \cite[p  193]{Engelking}. 

The following useful  Lemma is somehow a generalization of the fact that continuous function $f:X\to Y$ takes compact sets in $X$ to compact sets in $Y$.

\begin{lemma}\label{lem 0-2}
Let $X, Y$ be two topological spaces. If $f:X\to y$ is a continuous function and $A\subseteq X$ is $\lambda$-compact, then $d_c(f(A))\leq\lambda$.
\end{lemma}

\begin{proposition}\label{pro 0-3}
Let $F$ be a closed subset of topological space $X$, then $d_c(F)\leq d_c(X)$.
\end{proposition}

Next we are going to introduce a generalization of the concept $P$-space (Pseudo-discrete space). $P$-spaces are very important  in the contexts of rings of continuous functions which were introduced by  Gillman and Henriksen, see \cite{GH}.

\begin{definition}\label{def 0-4}
Let $X$ be a topological space. The intersection of any family with cardinality less than $\lambda$ of open sets of $X$ is called a $G_\lambda$-set. Obviously,  $G_\delta$-sets are precisely $G_{\aleph_1}$-sets.
\end{definition}

\begin{definition}\label{def 0-5}
The topological space $X$ is said to be a $P_\lambda$-space whenever each $G_\lambda$-set in $X$ is open.
It is manifest that every arbitrary space is a $P_{\aleph_0}$-space and $P_{\aleph_1}$-spaces are precisely $P$-spaces. See Example 1.8 and Example 1.9 in \cite{KNS}
\end{definition}

\section{$\lambda$-perfect maps}

\begin{lemma}\label{lemma-1}
If A is a $\beta$-compact subspace of a space $X$ and $y$ a point of a 
$P_{\lambda} $-space $Y$ such that $\beta\leq\lambda$  then for every open set 
$W\subseteq X\times Y$ containing $A\times\{y\}$ there exist open sets 
$U\subseteq X$ and $V\subseteq Y$ such that $A\times\{y\}\subseteq U\times 
V\subseteq W.$
\end{lemma}

\begin{proof}
For every $x\in A$ the point $(x,y)$ has a neighborhood of the form $U_x\times 
V_x$ contained in $W$. Clearly $A\times\{y\}\subseteq\bigcup_{x\in 
A}U_x\times V_x$. since $A$ is $\beta$-compact; there exists a family 
$\{x_i\}_{i\in I}$ with $|I|<\beta\leq\lambda$ such that 
$A\times\{y\}\subseteq\bigcup_{i\in I}U_{x_i}\times V_{x_i}$. Clearly the sets 
$U=\bigcup_{i\in I}U_{x_i}$ and $V=\bigcap_{i\in I}V_{x_i}$ have the required 
properties.
\end{proof}

Let us recall Kuratowski theorem.
\begin{theorem}\label{theorem-0-1}
For  a topological space $X$ the following conditions are equivalent:\\
(i) The space $X$ is  compact.\\
(ii) For every topological space $Y$ the projection $p:X\times Y\to Y$ is 
closed. \\
(iii) For every $T_4$-space $Y$ the projection $p:X\times Y\to 
Y$ is closed.
\end{theorem}
 Next  we  generalize the above theorem.

\begin{theorem}\label{theorem-1}
For an infinite cardinal 
number $\lambda$ and a topological space $X$ the following conditions are equivalent.\\
(i) The space $X$ is $\beta$-compact for $\beta\leq\lambda$. \\
(ii) For every $P_{\lambda} $-space $Y$ the projection $p:X\times Y\to Y$ is 
closed. \\
(iii) For every $T_4$, $P_{\lambda} $-space $Y$ the projection $p:X\times Y\to 
Y$ is closed.
\end{theorem}

\begin{proof}
(i)$\Rightarrow$(ii) Let $X$ be a $\beta$-compact space and $F$ be a closed 
subspace of $X\times Y$. Take a point $y\notin p(F)$; as 
$X\times\{y\}\subseteq(X\times Y)\setminus F$, it follows from the lemma \ref{lemma-1} 
that $y$ has a neighborhood $V$ such that $(X\times V)\bigcap F=\emptyset$. 
We then have $p(F)\bigcap V=\emptyset$, which shows that $p(F)$ is a closed 
subset of $Y$. \\
(ii)$\Rightarrow$(iii) It is  obvious\\
(iii)$\Rightarrow$ (i) Suppose that there exists a family 
$\{F_s\}_{s\in S}$ of closed subset of $X$ with the $\lambda$- intersection 
property such that $\bigcap_{s\in S}F_s=\emptyset$. Take a point $y_0\notin X$ 
and on the set $Y=X\bigcup\{y_0\}$ consider the topology $T$ consisting of all 
subsets of $X$ and of all sets of the form$$\{y_0\}\cup(\bigcap_{i\in 
I}F_{s_i})\cup K~~ \mbox{where}~~  s_i\in S~~ ,~~ |I|<\beta~~ 
\mbox{and}~~K\subseteq X. $$ 
 since $\bigcap_{s\in S}F_s=\emptyset$ for every $x\in X$ there exists $s_x\in S$ 
such that $x\notin F_{s_x}$. Notice that the set $$\{y_0\}\cup 
F_{s_x}\cup(X\setminus \{x\})=Y\setminus\{x\}$$ belong to $T$ 
and  since $X\in T$, $\{y_0\}$ is closed, therefor $Y$ is $T_1$-space.\\  If $A$ 
and $B$ are two disjoint closed subsets of $Y$ then at least one of them, say 
$A$, does not contain $y_0$ and so is open; therefor $A$ and $Y\setminus A$ are 
two disjoint open subset of $Y$  that contains $A$ and $B$ respectively. Hence 
$Y$ is normal.

 Take $F=\overline{\{(x,x):x\in X\}}\subseteq X\times Y$. By hypothesis $p(F)$ 
is closed in $Y$. Since $\{y_0\}$ is not open, every open subset of $Y$ that 
contains $y_0$ meets $X$. Hence $\overline X=Y$.\\
 
  If $x\in X$ then $(x,x)\in F$ therefor $x=p(x,x)\in p(F)$. Hence  $X\subseteq 
p(F)$ $$ \Rightarrow y_0\in Y=\overline X\subseteq\overline{p(F)}=p(F).$$ 
Therefor there exists $x_0\in X$ such that $(x_0,y_0)\in F.$ For every 
neighborhood $U\subseteq X$ of $x_0$ and every $s\in S$ the set 
$U\times(\{y_0\}\cup F_s) $ is open in $X\times Y$; hence 
$$[U\times(\{y_0\}\cup F_s)]\cap\{(x,x):x\in X\}\neq\emptyset,$$ i.e., $U\cap 
F_s\neq\emptyset$, which shows that $x_0\in F_s$ for every $s\in S$. This 
implies that $\bigcap_{s\in S}F_s\neq\emptyset$, and we have a contradiction.
 \end{proof}

\begin{definition}
Let $X$ and $Y$ be two topological spaces and $f:X\to Y$ be a continuous 
function.
We say that $f$ is $\lambda$-perfect  whenever $f$ be closed and $\lambda$ be the least infinite cardinal number such that for every $y\in Y$ the compactness degree of $f^{-1}(y)$ is less than or equal to $\lambda$ .
\end{definition}

\begin{theorem}
If $Y$ is a $P_{\lambda} $-space and $X$ is a $\beta$-compact space such that 
$\beta\leq\lambda$, then the projection $p:X\times Y\to Y$ is $\gamma$-perfect 
for $\gamma\leq\beta$.
\end{theorem}

\begin{proof}
Theorem \ref{theorem-1} implies that $p$ is closed. For every $y\in Y$ the set 
$p^{-1}(y)$ 
is of the form $A\times\{y\}$ for some $A\subseteq X$ and it is closed in $X\times Y$, 
by the continuity of $p$, So it is closed in $X\times\{y\}$ and since this set is $\beta$-
compact, we infer that $p^{-1}(y)$ is $\alpha$-compact for $\alpha\leq\beta$ by \cite[Proposition 2.6]{NS}. 
Consequently $p$ is $\gamma$-perfect for $\gamma\leq\beta$.
\end{proof}

The following theorem is an extension of Theorem 3.7.2 in \cite{Engelking}
\begin{theorem}
Let $\lambda$ and $\beta$ be two regular cardinal numbers and $f:X\to Y$ be a $\lambda$-perfect mapping, then for every $\beta$-compact 
subspace $Z\subseteq Y$ the inverse image $f^{-1}(Z)$ is  $\mu$-compact for 
$\mu\leq\max\{\lambda , \beta\}$.
\end{theorem}

\begin{proof}
Let $\{U_s\}_{s\in S}$ be a family of open subsets of $X$ whose union contains 
$f^{-1}(Z)$, $K$ be the family of all subsets of $S$ with cardinality less than 
$\lambda$ and $ U_T=\bigcup_{s\in T}U_s$ for $T\in K$. For each $z\in Z$, $d_{c}(f^{-1}(z))\leq\lambda$ and thus is contained in the set $U_T$ for 
some $T\in K$; it follows that $z\in Y\setminus f(X\setminus U_T)$, and thus:
$$Z\subseteq\bigcup_{T\in K}(Y\setminus f(X\setminus U_{T})).$$ Since $f$ is 
closed; the sets $Y\setminus f(X\setminus U_{T})$ being open. Hence there 
exists 
$I\subseteq K$ with $|I|< \beta$ such that $Z\subseteq\bigcup_{T\in I}
(Y\setminus f(X\setminus U_{T}))$. Hence 
$$f^{-1}(Z)\subseteq\bigcup_{T\in I}f^{-1}(Y\setminus f(X\setminus 
U_T))=\bigcup_{T\in I}(X\setminus f^{-1}f(X\setminus 
U_T))\subseteq$$$$\bigcup_{T\in I}(X\setminus(X\setminus 
U_T))=\bigcup_{T\in I}U_T=\bigcup_{s\in S_0}U_s,$$ where 
$S_0=\bigcup_{T\in I}T.$\\
We notice that $S_0$ is the union of a family with cardinality less than $\beta$
of member of $K$ which has cardinality less than $\lambda$. Since $\lambda$ 
and $\beta$ are regular cardinals; $|S_0|<\max\{\lambda , \beta\}$, i.e., 
$f^{-1}(Z)$ is $\mu$-compact for $\mu\leq\max\{\lambda , \beta\}$.
\end{proof}

The following corollary is now immediate.
\begin{corollary}
If $f:X\to Y$ is an onto $\lambda$-perfect mapping, then for every $\lambda$-
compact subspace $Z\subseteq Y$ the inverse image $f^{-1}(Z)$ is  $\lambda$-
compact too.
\end{corollary}

\begin{proof}
By previous theorem $f^{-1}(Z)$ is $\mu$-compact for $\mu\leq\lambda$; and 
by (  ) $Z=f(f^{-1}(Z))$ is $\gamma$-compact for $\gamma\leq\mu$. But by 
hypothesis $\gamma=\lambda$, so, $\mu=\lambda$.
\end{proof}

\begin{corollary}
 If $g:X\to Z$ be a $\lambda$-perfect and $f:Z\to Y$ be a $\beta$-perfect 
mapping then the composition $fg:X\to Y$ is $\gamma$-perfect for 
$\gamma\leq\max\{\lambda , \beta\}$.
\end{corollary}

It is well-known that in every regular space we can separate compact sets  from closed sets by the open sets. Next we have an extension of this fact.

\begin{theorem}
If $A$ is a $\beta$-compact subspace of a regular $P_{\lambda}$-space $X$, for $\beta\leq\lambda$, then for every closed set $B$ disjoint from $A$ there exist open sets $U , V\subseteq X$ such that $A\subseteq U$ , $B\subseteq V$ and $U\cap V=\emptyset$.\\ If, moreover, $d_c(B)=\alpha\leq\lambda$ then it suffices to assume that $X$ is a Hausdorff $P_{\lambda}$-space .
\end{theorem}

\begin{proof}
Since $X$ is regular, for every $x\in A$ there exist open sets$~U_x ,~ V_x\subseteq X$ such that $x\in U_x~~,~~B\subseteq V_x$ and \\$ U_x\cap V_x=\emptyset$. Clearly $A\subseteq\bigcup_{x\in A}U_x$, therefor there exists $J\subseteq A$ whit $|J|\leq \beta$ such that $A\subseteq\bigcup_{x\in J}U_x$. Obviously the sets $U=\bigcup_{x\in J}U_x$ and $V=\bigcap_{x\in J}V_x$ have the required properties. Now let we assume that $d_c(B)=\alpha\leq\lambda$ and $X$ is only a Hausdorff $P_{\lambda}$-space. For every $x\in B$ we consider the sets $A$ and $\{x\}$ and clearly similar to the first part we will obtain the disjoint open sets $U_x$ and $V_x$ such that $x\in V_x$ and $A\subseteq U_x$.  Clearly $B\subseteq\bigcup_{x\in B}V_x$
\end{proof}
\begin{lemma}
If $X$ is  a Hausdorff $P_{\lambda} $-space then a $\beta$-perfect mapping 
$f:X\to Y$, whenever $\beta\leq\lambda$, can not be continuously extended over 
any Hausdorff space $Z$ that contains $X$ as a dense proper subspace.
\end{lemma}

\begin{proof}
 Without loss of generality we can suppose that $Z=X\cup\{x\}$. Now if $F:Z\to 
Y$ is a continuous extension of $f$ then $d_c(f^{-1}(F(x)))\leq\beta$ and clearly $f^{-1}(F(x))$ dose not contain $x$. Therefor by the second part of 
theorem 0.9. there exist open sets $U,V\subseteq Z$ such that $x\in U,f^{-1}
(F(x))\subseteq V$ , and $U\cap V=\emptyset$. Clearly the sets $(Z\setminus 
V)\cap X=X\setminus V$ ,   $f(X\setminus V)$ and $F^{-1}(f(X\setminus V))$ 
are closed.  Thus we have $$cl_z(X\setminus V)\subseteq F^{-1}(f(X\setminus 
V))=f^{-1}(f(X\setminus V))\subseteq X.$$ Since $x\notin cl_zV$, we have 
$cl_zV\subseteq X$ therefor $cl_zX=cl_z(X\setminus V)\cup cl_zV\subseteq X$. 
Hence $X$ does not dense in $Z$.
\end{proof}

\begin{proposition}
 If the composition $gf$ of continuous mappings $f:X\to Y$ and $g:Y\to Z$ is 
$\lambda$-perfect then the mappings $g_f=g|f(X)$ and $f$ are $\beta$-perfect 
and $\mu$-perfect respectively for $\mu , \beta\leq\lambda$.
\end{proposition}

\begin{proof}
Since for every point $z\in Z$ we have $d_{c}((gf)^{-1}(z))\leq\lambda$; it 
follows from proposition () that the compactness degree of the set $g_f^{-1}(z)=f(X)\cap g^{-1}
(z)=f[(gf)^{-1}(z)]$ is less than or equal to $\lambda$.  The fact that 
$g_f$ is a closed mapping follows from [proposition 2.1.3  Eng] and thus the 
mapping $g_f$ is $\mu$-perfect for $\mu\leq\lambda$.\\
 We note that for every point $y\in Y$, $f^{-1}(y)$ is closed subset of $(gf)^{-1}(g(y))$ then $d_c(f^{-1}(y))\leq d_c(gf)^{-1}(g(y))\leq\lambda$. thus to conclude the proof it suffices to show that $f$ is closed. For this end we consider the arbitrary closed closed set $F\subseteq X$ and the mapping $h=(gf)|F: F\to Z.$ For every $z\in Z, h^{-1}(z)=(gf)^{-1}(z)\cap F$ is closed subset of $(gf)^{-1}(z)$ then $d_c(h^{-1}(z))\leq d_c((gf)^{-1}(z))\leq\lambda.$ this means that $h$ is $\alpha$-perfect for $\alpha\leq\lambda.$ so that, by the first part of our proof, the restriction $g|f(F)$ is $\theta$-perfect for $\theta\leq\alpha$. Clearly $g|f(F)$ can be continuously extended over $\overline{f(F)}.$ (note that the restriction of $g$ to $f(F)$ is the continuous extension of $g|f(F)$ over $\overline {f(F)}.$) It follows from the previous lemma that $f(F)=\overline {f(F)}$ and thus $f$ is a closed mapping.
\end{proof}

\end{document}